\documentclass[12pt,reqno]{article}

\usepackage[usenames]{color}
\usepackage{amssymb}
\usepackage{amsmath}
\usepackage{amsthm}
\usepackage{amsfonts}
\usepackage{amscd}
\usepackage{graphicx}

\usepackage[colorlinks=true,
linkcolor=webgreen,
filecolor=webbrown,
citecolor=webgreen]{hyperref}

\definecolor{webgreen}{rgb}{0,.5,0}
\definecolor{webbrown}{rgb}{.6,0,0}

\usepackage{color}
\usepackage{fullpage}
\usepackage{float}

\usepackage{graphics}
\usepackage{latexsym}
\usepackage{epsf}
\usepackage{breakurl}

\newcommand{\seqnum}[1]{\href{https://oeis.org/#1}{\rm \underline{#1}}}

\newcommand{\N}{{\mathbb N}}

\def\gcd{\mathrm{gcd}}
\def\modd#1 #2{#1\ \mbox{\rm (mod}\ #2\mbox{\rm )}}

\begin{document}

\begin{center}
\epsfxsize=4in

\end{center}

\theoremstyle{plain}
\newtheorem{theorem}{Theorem}
\newtheorem{corollary}[theorem]{Corollary}
\newtheorem{lemma}[theorem]{Lemma}
\newtheorem{proposition}[theorem]{Proposition}

\theoremstyle{definition}
\newtheorem{definition}[theorem]{Definition}
\newtheorem{example}[theorem]{Example}
\newtheorem{conjecture}[theorem]{Conjecture}
\newtheorem{problem}[theorem]{Problem}

\theoremstyle{remark}
\newtheorem{remark}[theorem]{Remark}

\begin{center}
\vskip 1cm{\LARGE\bf  Integer Sequences which Are Closed with Respect to Multiplication and whose Sumset does not Intersect the Sequence
}

\vskip 1cm
\large
Georgi Demirov\\
Independent Scholar\\
9000 Varna\\
Bulgaria\\
\href{mailto:gdemirov1954@gmail.com}{\tt gdemirov1954@gmail.com} \\
\vskip .2 in
\large
Ognian Trifonov\\
University of South Carolina\\
Department of Mathematics \\
Columbia, SC 29208 \\
USA\\
\href{mailto:trifonov@math.sc.edu}{\tt trifonov@math.sc.edu} \\
\end{center}

\vskip .2 in

\begin{abstract}
We investigate increasing sequences of integers with the property that the product of every two terms of the sequence is also a term of the sequence and the sum of every two terms is not a term of the sequence. We say that a sequence with the above properties is maximal if it is not a proper subsequence of another sequence with the above properties. We determine whether the sequences in three families are maximal or not.
\end{abstract}

\section{Introduction}
Let $S = (a_n )_{n \geq 1}$ be an increasing sequence of positive integers which satisfies the following two properties
\begin{itemize}
\item ${\mathcal P}_1$ For every $n \in \N $ and $m \in \N$, we have $a_m a_n \in S$; and
\item ${\mathcal P}_2$  For every $n \in \N $ and $m \in \N$, we have $a_m +  a_n \not  \in S$.
\end{itemize}
In this paper we have adopted the convention that if $S$ is a sequence, then $a \in S$ means that $a$ is one of the terms of the sequence. 

For brevity, we refer to increasing sequences of positive integers which satisfy ${\mathcal P}_1$ and ${\mathcal P}_2$ as {\it GD sequences}. 

Below are some examples of GD sequences. 

\begin{example} \label{Ex1}
Let $k \in \mathbb{N}$. Define the sequence $T_k = (n^k)_{n\geq 1}$. Then, $T_k$ is a GD sequence if $k \geq 3$. Clearly, property ${\mathcal P}_1$ holds. Also, 
property ${\mathcal P}_2$ holds for $k \geq 3$ due to the monumental work of Wiles \cite{Wiles}.  Note that sequences $T_3$, $T_4$, $T_5$, and $T_6$ are sequences \seqnum{A000578}, \seqnum{A000583}, \seqnum{A000584}, and 
\seqnum{A001014} in the OEIS \cite{oeis}, respectively.
\end{example}

\begin{example} \label{Ex2}
Let $k \in \mathbb{N}$. Define the sequence $V_k =((n-1)k+1)_{n\geq 1}$. Then, $V_k$ is a GD sequence if $k \geq 2$. It is easy to see that  both properties ${\mathcal P}_1$  and ${\mathcal P}_2 $ hold. The 
sequences  $V_3$, $V_4$, $V_5$, and $V_6$ are sequences \seqnum{A005408}, \seqnum{A016777}, \seqnum{A016813}, and 
\seqnum{A016861} in the OEIS \cite{oeis}, respectively.
\end{example}

\begin{example} \label{Ex3}
Let $k \in \mathbb{N}$. Define the sequence $W_k =(k^{n-1})_{n\geq 1}$. Then, $W_k$ is a GD sequence if $k \geq 3$. Indeed, property ${\mathcal P}_1$ holds and property ${\mathcal P}_2$ holds as well, since $k^n + k^n < k^{n+1}$ if $k >2$. 
The sequences  $W_3$, $W_4$, $W_5$, and $W_6$ are sequences \seqnum{A000244}, \seqnum{A000302}, \seqnum{A000351}, and 
\seqnum{A000400} in the OEIS \cite{oeis}, respectively.
\end{example}

The notion of a GD sequence can be extended by allowing  its terms to be real numbers, complex numbers, or elements of a ring. For example, the sequence $(x^{n-1})_{n \geq 1}$ satisfies both properties 
${\mathcal P}_1$ and ${\mathcal P}_2$ if $x$ is a real number such that $x > \frac{1+\sqrt{5}}{2}$ and $x \neq 2$. Also, the sequence $(z^{n-1})_{n \geq 1}$ satisfies both properties 
${\mathcal P}_1$ and ${\mathcal P}_2$ if $z$ is a complex number such that $|z| > \frac{1+\sqrt{5}}{2}$ and $z \neq 2$, and so does the finite sequence $(\hat{1}, \hat{3})$ where $\hat{1}$ and $\hat{3}$ are elements of $\mathbb{Z}_4$.

We  say that a GD sequence $S$ is a {\it maximal GD sequence} if there is no GD sequence $S_1$ such that $S$ is a proper subsequence of $S_1$. For example, the sequence $V_2$ is a maximal GD sequence. 

In this paper we determine which of the GD sequences defined in Example \ref{Ex1}, Example \Ref{Ex2}, and Example \ref{Ex3} are maximal and which are not. 

The main results of the paper are the following two theorems.

\begin{theorem} \label{T1}
For every integer $k \geq 3$ the GD sequence $T_k$ is not maximal.
\end{theorem} 

\begin{theorem}  \label{T2}
We have
\begin{itemize}
\item the GD sequences $V_k$ are maximal for $k=2,3$, but are not maximal for  $k \geq 4$; and

\item the GD sequences $W_k$ are not maximal for all $k \geq 3$.
\end{itemize}
\end{theorem}

The proof of Theorem \ref{T1} depends in a crucial way on the following theorem, Theorem \ref{T3}. Theorem \ref{T3} is not original; rather, it is a compilation of  results of several authors.

 \begin{theorem}[Cohen, Fermat, Ratcliffe, Selmer,  Wiles, and Xuan] \label{T3}
 
 Let $a$ be a nonnegative integer. Then, 
 \begin{itemize}
 \item for every odd prime $p$, the Diophantine equation 
 \begin{equation} \label{Eqp}
 x^p + y^p = 11^a z^p
 \end{equation}
 has no solutions in integers $x$, $y$, and $z$ such that $xyz \neq 0$; and
 \item
 the Diophantine equations 
 \begin{equation} \label{Eq4}
 x^4 \pm y^4 = 11^a z^4
 \end{equation}
  have no solutions in integers $x$, $y$, and $z$ such that $xyz \neq 0$.
 \end{itemize}
 \end{theorem}
 
 We conclude the introduction by stating an open problem.
 
 \begin{problem} \label{Pr1}
 For each of the GD sequences defined in Example \ref{Ex1}, Example \Ref{Ex2}, and Example \ref{Ex3} construct a maximal GD sequence containing it.
 \end{problem} 

The above problem appears to be simpler for sequences in the family $V_k$, $k \geq 2$. For such sequences we work with residue classes modulo $k$, so it is not difficult to find a maximal GD
sequence containing $V_k$  when $k$ is relatively small. Also, if $p$ is prime and $p \mid k$, then $V_k$ is a subsequence of $V_p$. Thus, it is sufficient to find maximal GD sequences containing $V_p$  for primes $p \geq 5$. We do not have a formula which works for all primes  $p$.

\section{Proofs of the Theorems}
 
 First, we prove Theorem \ref{T3}. 
 
 \begin{proof}
 (i) We can assume that $a<p$. Indeed, let $a = pq + r$ for some nonnegative integers $q$ and $r$ with $r <a$. Then, equation \eqref{Eqp} can be rewritten as
 $$x^p + y^p = 11^r (11^qz)^p.$$
 
 Next, if $a=0$, equation \eqref{Eqp} becomes  $x^p+y^p=z^p$ and was solved by Wiles \cite{Wiles}. Thus, we can assume $1 \leq a \leq p-1$. 
 
 Next, we consider the case $p=3$. In 1951, Selmer \cite{Selmer} published an extensive study on the cubic Diophantine equation $ax^3 + by^3 +cz^3 = 0$. Among many other things he showed that the Diophantine 
 equations $x^3 + y^3 = 11z^3$ and $x^3 + y^3 = 121z^3$ have no integer solutions with $xyz \neq 0$. (See \cite{Selmer} Table IV on p.352.)  Thus, the first part of the theorem holds when $p=3$.
 
 For the case $p \geq 5$ we use the result of Cohen \cite{Cohen} Theorem 15.5.3 which states 
 
 {\it Suppose that $3 \leq s \leq 100$ is prime, $s \neq 31$.  Then the equation  $x^p + y^p = s^\alpha z^p$ with prime $p \geq 5$ and any $\alpha \geq 1$ does not have any solutions with $x, y, z$ nonzero and pairwise coprime.}

One obstacle to using the above result is the requirement $x,y,z$ to be pairwise coprime. 
However, we can show that if $x,y,z$ are nonzero integers satisfying equation \eqref{Eqp} for some integer $a$, $1 \leq a \leq p-1$, then there exists 
nonzero pairwise coprime integers $x_1, y_1, z_1$ satisfying equation \eqref{Eqp} with the same $a$. 

Let $x,y,z$ be nonzero integers satisfying equation \eqref{Eqp} for some integer $a$ between $1$ and $p-1$. We can assume that $\gcd(x,y,z)=1$. If $\gcd(x,y,z)=d > 1$, then $x/d, y/d$, and $z/d$ satisfy equation \eqref{Eqp} and
$\gcd (x/d, y/d,z/d)=1$. Next, $11 \nmid x$. Otherwise, $11 \mid y^p = 11^az^p - x^p$, so $11 \mid y$, Thus, $11^p | x^p + y^p = 11^az^p $, and $p \mid z$ (since $a < p$), which contradicts $\gcd(x,y,z)=1$. 
Similarly, $11  \nmid y$. We claim that $x,y,z$ are pairwise coprime. Assume not. First, consider the case $\gcd(x,y) = d_1 > 1$. Let $q$ be a prime divisor of $d_1$. Then, $q \neq 11$ and $q \mid x^p + y^p = 11^a z^p$, so $q \mid z$. 
Therefore, $q \mid \gcd(x,y,z)=1$, a contradiction. We proceed similarly in the cases $\gcd(x,z) > 1$ and $\gcd(y,z)>1$. 

Applying Cohen's result \cite{Cohen} with $s=11$ and $p \geq 5$ we obtain that there are no nonzero pairwise coprime integers $x_1, y_1,z_1$ which satisfy \eqref{Eqp} with $1 \leq a \leq p-1$. This completes the proof of part (i) of the theorem.

(ii)  Proceeding as in part (i), we can assume that in both Diophantine equations \eqref{Eq4}, we have $a=0,1,2,$ or $3$.

When $a=0$ Fermat proved that the equations  $$x^4 \pm y^4 = z^4$$ have no solutions in nonzero integers $x,y,z$. 

So, from now on, we consider the case $a=1,2,3$ only. 

First, we consider the Diophantine equation 

\begin{equation} \label{Eq4+}
x^4 + y^4 = 11^a z^4,
\end{equation}
with $a=1,2$, or $3$.

Assume that $(x,y,z)$ is a solution in nonzero integers of equation \eqref{Eq4+}. As before, we can assume that $\gcd (x,y,z)=1$. (If $\gcd(x,y,z)=d >1$, consider $x/d, y/d, z/d$.)

Next we use two well-known facts from elementary number theory. 

Fermat's theorem {\it If $a$ is an integer and $p$ is a prime such that $p \nmid a$, then $a^{p-1} \equiv \modd{1} {p}$}. Fermat's theorem implies 
{\it If $p \equiv \modd{3} {4}$ is a prime such that $p \mid x^2 + y^2$ for some integers $x$ and $y$, then $p \mid x$ and $p \mid y$.} (For a proof assume that $p \nmid x$ and raise the congruence
$x^2 \equiv \modd{-y^2} {p}$ to power $(p-1)/2$ to get a contradiction.)

Now, $11 \mid 11^a z^4 = (x^2)^2 + (y^2)^2$, so $11 \mid x$ and $11 \mid y$. Therefore, $11^4 \mid x^4 + y^4 = 11^a z^4$, thus $11 \mid z$. This contradicts $\gcd(x,y,z)=1$, so equation \eqref{Eq4+} has no solutions in nonzero integers $x,y,z$. 

Next, we consider the Diophantine equation 

\begin{equation} \label{Eq4-}
x^4 - y^4 = 11^a z^4,
\end{equation}
with $a=1,2$, or $3$.

If $x,y,z$ are nonzero integers satisfying equation \eqref{Eq4-} then the equation can be rewritten in the form $(x/z)^4 - (y/z)^4 = 11^a$.

Very recently, in April 2026, A.~Ratcliffe and N.~Tho \cite{Ratcliffe} determined all positive integers $\leq 10000$ which can be represented as a difference of fourth powers of two rational numbers. Table 3 on p.14 of the paper \cite{Ratcliffe} contains all positive integers up to $10000$ which are difference of two fourth powers of rational numbers, together with one representation for each integer of that form. The integers $11$, $121$, and $1331$ are not in the above table.

Therefore, equation \eqref{Eq4-} has no solutions in nonzero integers $x,y,z$, which completes the proof of the theorem.
\end{proof}

A quick observation from Table 3 on p.14 of the paper \cite{Ratcliffe} is that the second Carmichael number $1105=7^4 - 6^4$ and the fourth Carmichael number $2465=9^4-8^4$ are the only Carmichael numbers up to $10000$ which are the difference of two fourth powers of rationals. Are there other Carmichael numbers which are the difference of two fourth powers of rationals (or integers)?

Next, we prove Theorem \ref{T1}
\begin{proof}
We need to show that $T_k$ is not a maximal GD sequence when $k \geq 3$. 

Consider the case when $k$ is a composite integer and $k \neq 4$. Let $p$ be the least prime divisor of $k$. Then, $k=pl$ for some integer $l \geq 3$. (If $l=2$, then $p=2$ and $k=4$.)

In this case $T_k$ is a proper subsequence of the GD sequence $T_l$, so it is not a maximal GD sequence. 

Next, let $k=p$, where $p$ is an odd prime. Let $T'_p$ be the sequence whose terms are the elements of the set $\{ 11^a n^p\  |\ n \in \mathbb{N}, a \in \mathbb{Z}, a \geq 0 \}$  arranged in increasing order. 

Clearly, $T_p$ is a proper subsequence of the sequence $T'_p$. We claim that $T'_p$  is a GD sequence. The sequence $T'_p$ does satisfy property ${\mathcal P}_1$. To prove that it satisfies property ${\mathcal P}_2$ we show that the Diophantine equation 
\begin{equation} \label{Eqp11}
11^a x^p + 11^b y^p = 11^c z^p
\end{equation}
has no solutions with $x,y,z$ nonzero integers and $a,b,c$ nonnegative integers.

Let  $x,y,z$ be nonzero integers and let $a,b,c$ be nonnegative integers which satisfy equation  \eqref{Eqp11}. We can assume that $11 \nmid x$, $11 \nmid y$, and $11 \nmid z$. If $x =11^q x_1$ where $11 \nmid x_1$ we replace $a$ by $a + pq$ and $x$ by $x_1$ and proceed similarly with $y$ and $z$. Next, we can assume $a = \min (a,b,c)$. Otherwise, if necessary,  we relabel $x,y,z$, or, to move a term from one side of the equation to the other, replace $x$ by $-x$, or $y$ by $-y$, or $z$ by $-z$,  and relabel. 
Dividing by $11^a$ and relabeling the exponents of $11$ we obtain 
\begin{equation} \label{Eqp11s}
 x^p + 11^b y^p = 11^c z^p.
\end{equation}

Note that in equation \eqref{Eqp11s} either $b=0$ or $c=0$. Otherwise, $11 \mid 11^c z^p - 11^b y^p = x^p$, so $11 \mid x$, a contradiction. 

We can assume $b=0$ (otherwise, replace $y$ by $-y$, $z$ by $-z$,  and relabel. )

Thus, equation \eqref{Eqp11s} further reduces to the equation $$x^p + y^p = 11^c z^p,$$ which,  by Theorem \ref{T3} has no solutions in nonzero integers $x,y,z$. 

Therefore, $T'_p$ is a GD sequence. 

Finally, let $k=4$. Let $T'_4$ be the sequence whose terms are the elements of the set $\{ 11^a n^4\  |\ n \in \mathbb{N}, a \in \mathbb{Z}, a \geq 0 \}$  arranged in increasing order. 

Again, $T_4$ is a proper subsequence of the sequence $T'_4$. We claim that $T'_4$ is a GD sequence. Property ${\mathcal P}_1$ holds for $T'_4$ and we need to establish that property ${\mathcal P}_2$ holds  as well. We claim that
 the Diophantine equation 
\begin{equation} \label{Eq411}
11^a x^4 + 11^b y^4 = 11^c z^4
\end{equation}
has no solutions with $x,y,z$ nonzero integers and $a,b,c$ nonnegative integers.

Let  $x,y,z$ be nonzero integers and let $a,b,c$ be nonnegative integers which satisfy equation \eqref{Eq411}. 
Proceeding as in the previous argument we can assume $11 \nmid x$, $11 \nmid y$, and $11 \nmid z$. Let $m = \min(a,b,c)$. Dividing equation \eqref{Eq411} by $11^m$  and relabeling the exponents of $11$ we conclude that at least one of $a,b,c$ equals $0$. 
As above, we cannot have exactly one of $a,b,c$ to equal $0$. For example, if $a=0$, $b > 0$, $c>0$ equation \eqref{Eq411} becomes
$$x^4 + 11^b y^4 = 11^c z^4.$$
In this case, $11 \mid 11^c z^4 - 11^b y^4 = x^4$, so $11 \mid x$, a contradiction. We proceed similarly in the cases $b=0$, $ac > 0$ and $c=0$, $ab >0$. Therefore, equation \eqref{Eq411} reduces to one of, 
$$x^4 + y^4 = 11^c z^4, \quad x^4 + 11^b y^4 = z^4, \text{ or } 11^ax^4 + y^4 = z^4.$$

After relabeling all three equations above become one of the equations 
$$x^4 \pm y^4 = 11^a z^4,$$
which is equation \eqref{Eq4}.

We showed in Theorem \ref{T3} (ii) that the Diophantine equation \eqref{Eq4} has no solutions in nonzero integers $x,y,z$ when $a$ is nonnegative integer. Thus, equation \eqref{Eq411} has no solution 
with $x,y,z$ nonzero integers and $a,b,c$ nonnegative integers, completing the proof of the theorem.

\end{proof}

Last, we prove Theorem \ref{T2}.
\begin{proof}
(i) The sequence $V_2$ (all odd natural numbers) is a maximal GD sequence. If $V_2$ is a proper subsequence of a sequence $V'_2$ of positive integers, then $V'_2$ contains an even integer and property ${\mathcal P}_2$ does not hold for $V'_2$. 

Similarly, $V_3$ is a maximal GD sequence. Indeed, let $V_3$ be a proper subsequence of the sequence of positive integers $V'_3$. Let $a$ be a term in $V'_3$ which is not in $V_3$. If $a \equiv \modd{0} {3}$, then $1 + a \in V_3$, so 
property ${\mathcal P}_2$ does not hold for $V'_3$. If $a \equiv \modd{2} {3}$, then $a-1 \in V_3$, and $1 + (a-1)=a$, so 
property ${\mathcal P}_2$ does not hold for $V'_3$. 

We claim that if  $k \geq 4$, then $V_k$ is not a maximal GD sequence. Indeed, let $V'_k$ be the sequence whose terms are the elements of the set $\{  n \in \mathbb{N}\ | \ n \equiv \modd{\pm 1} {k} \}$ arranged in increasing order. Then, $V_k$ is a proper subsequence of $V'_k$ and $V'_k$ is a GD sequence.

(ii) We need to show that the sequence $W_k =(k^{n-1})_{n\geq 1}$ where $k \geq 3$ is not a maximal GD sequence.  Let $p$ be a prime such that $p > k+1$.  Let $W'_k$ be the sequence with terms the elements of the set  
$\{ p^{2(m-1)} k^{n-1} | m \in \mathbb{N}, n \in \mathbb{N} \}$ arranged in increasing order. Clearly, $p \nmid k$, so $W_k$ is a proper subsequence of $W'_k$. We claim that $W'_k$ is a GD sequence.  Property ${\mathcal P}_1$ clearly holds. 
To show that property ${\mathcal P}_2$ holds for $W'_k$ we prove that the Diophantine equation
\begin{equation} \label{Eqabc}
p^{2d} k^{a} + p^{2f}k^b = p^{2g} k^c
\end{equation}
 has no solutions in nonnegative integers $a,b,c,d,f,g$. 
 
 Assume that $a,b,c,d,f,g$ are nonnegative integers which satisfy equation \eqref{Eqabc}. We can assume that at least one of $d,f,g$ equals $0$. Indeed,  let $n = \min(d,f,g)$. Dividing equation \eqref{Eqabc} by $p^{2n}$ and relabeling we  obtain that at least one of the exponents $d,f,g$ is $0$. Next, it is not possible for exactly one of the exponents $d,f,g$ to be $0$. For example, if $d=0$ and $fg > 0$ the equation \eqref{Eqabc} has the form $$k^a + p^{2f}k^b = p^{2g} k^c.$$
 Thus, $p \mid  p^{2g} k^c - p^{2f}k^b = k^a$, so $p \mid k$. This is impossible since  $p > k$. We proceed similarly if $f=0$ and $dg > 0$ and when $g=0$ and $df > 0$. 
 
 Therefore, two or three of the exponents $d,f,g$ are $0$. 
 
 After relabeling we are left with three Diophantine equations to consider. 
 
 These are
 \begin{equation} \label{Eq3}
k^a + k^b = k^c, \quad p^{2d}k^a + k^b = k^c, \text{ and } k^a + k^b = p^{2d}k^c
\end{equation}

Since $k \geq 3$ the Diophantine equation $k^a + k^b = k^c$ has no solutions in nonnegative integers. (We need $c > a$ and $c>b$, so $k^c > 2k^{c-1} \geq k^a + k^b$.)

After relabeling, we are left with the two Diophantine equations
\begin{equation} \label{Eq2}
k^a \pm k^b = p^{2d}k^c.
\end{equation}

Same as above, we can assume that at least one of $a,b,c$ is zero. Again, it is not possible for exactly one of $a,b,c$ to be $0$. For example, if $c = 0$ and $ab > 0$ we have 
$k^a \pm k^b = p^{2d}$, so $k \mid p^{2d}$ which is impossible since $1 < k < p$ and $p$ is prime. The remaining cases $a=0$, $bc>0$ and  $b=0$, $ac >0$ are analogous. 

When $a=b=0$ we get $1 \pm 1 = p^{2d}k^c$ which is impossible since $p>k \geq 3$ and if $k=c=0$, $1 \pm 1 \neq 1$. 

Thus, at least one of $a,b$ is  nonzero, so $c=0$. 

Relabeling we get the two Diophantine equations 
\begin{equation}  \label{EqCat}
p^{2d} - k^a = 1 \text{ and } k^a - p^{2d} = 1.
\end{equation}

Note that in the above two equations we cannot have $d=0$ since $k \geq 3$. So, $d \geq 1$. 

Also, $a \geq 2$ since $p > k+1$. 

Now, we  use the famous Catalan's conjecture (also known as  Mih\u{a}ilesku's theorem).  It states

{\it  The only solution in the natural numbers of $$x^a - y^b = 1$$
for $a,b >1$, $x,y>0$ is $x=3$, $a=2$, $y=2$, $b=3$.}

Following the work of many mathematicians before him,  Mih\u{a}ilesku \cite{Mih}  proved the Catalan's conjecture in 2002. 

Since $p \geq 5$, equations \eqref{EqCat} have no solutions in nonnegative integers, thus \eqref{Eqabc} has no solutions in nonnegative integers $a,b,c,d,f,g$, completing the proof of the theorem.

\end{proof}

\section{ Acknowledgments} The second author  appreciates the help of Michael Bennett on the state of the art of several Diophantine equations.

\bigskip
\hrule
\bigskip

\noindent 2020 {\it Mathematics Subject Classification}:
Primary 11B83, Secondary 11D41, 11D72.

\noindent \emph{Keywords: }
integer sequences, generalized Fermat's equation, Catalan's conjecture.

\bigskip
\hrule
\bigskip

\noindent (Concerned with sequences
\seqnum{A000244},
\seqnum{A000302},
\seqnum{A000351},
\seqnum{A000400},
\seqnum{A000578},
\seqnum{A000583},
\seqnum{A000584},
\seqnum{A001014},
\seqnum{A005408},
\seqnum{A016777},
\seqnum{A016813}, and
\seqnum{A016863}.)

\end{document}